\documentclass{amsart}
\usepackage{geometry} 

\usepackage{tikz}
\usetikzlibrary{arrows,positioning} 
\usepackage{tikz-cd}
\usepackage[utf8]{inputenc}
\usepackage[english]{babel}
\usepackage{graphicx}
\usepackage{amssymb}
\usepackage{epstopdf}
\usepackage{enumerate}
\usepackage{mathtools}
\usepackage{setspace}
\usepackage{mathrsfs}
\usepackage{graphicx}
\usepackage{amsthm}
\usepackage{fancyhdr}
\usepackage{thmtools}
\numberwithin{equation}{section}
\graphicspath{{./LaTeX_Images/}}
\theoremstyle{plain}
\newtheorem{thm}{Theorem}[section]
\newtheorem{prop}[thm]{Proposition}
\newtheorem{lem}[thm]{Lemma}
\newtheorem{rem}[thm]{Remark}
\newtheorem{cor}[thm]{Corollary} 
\theoremstyle{definition}
\newtheorem{defn}[thm]{Definition}
\newtheorem{exmp}{Example}[section]
\DeclareSymbolFont{myletters}{OML}{ztmcm}{m}{it}
\DeclareMathSymbol{\uplambda}{\mathord}{myletters}{"15} 
\usepackage[unicode=true]{hyperref}
\hypersetup{
    colorlinks=true,
    linkcolor=blue,
    citecolor=purple,
    filecolor=magenta,      
    urlcolor=cyan,
    pdftitle={Stable Pairs on the Projective Plane},
    pdfpagemode=FullScreen,
    }

\title{Stable Wild Vafa-Witten Bundles on the Projective Plane}

\author{Robert J. Cornea}

\address{Department of Pure Mathematics, University of Waterloo, 200 University Ave W, Waterloo, ON N2L 3G1, Canada}

\email{rcornea@uwaterloo.ca}

\thanks{The author would like to thank their supervisor Ruxandra Moraru for the countless helpful discussions and patience, in which without, this work could not have been done. The author would also like to thank Steven Rayan for help and motivation of this work.}

\keywords{Vafa-Witten Bundles, Schwarzenberger Bundles, Spectral Correspondence, Deformation Theory, Moduli Spaces}

\begin{document}

\begin{abstract}
    This work explores the geometry of stable wild Vafa-Witten bundles over the complex projective plane $\mathbb{P}^2$. Specifically, we consider stable rank-two pairs $(E,\Phi)$, with $E\to\mathbb{P}^2$ a rank-two holomorphic vector bundle and $\Phi\in H^0(\mathbb{P}^2,\mathrm{End}_0E\otimes\mathcal{O}(d))$ for $d\geq0$, and compute the dimension of the moduli space of such stable pairs. Moreover, we classify stable pairs $(E,\Phi)$ when the underlying rank-two bundle $E$ splits or is the push-forward of a line bundle on $\mathbb{P}^1\times\mathbb{P}^1$. Lastly, we examine the fixed point locus of the natural $\mathbb{C}^*$-action on the moduli space.
\end{abstract}

\maketitle 
\newcommand{\op}[1]{\operatorname{\boldsymbol{\mathrm{#1}}}}
\newcommand{\R}{\mathbb{R}}
\newcommand{\C}{\mathbb{C}}
\newcommand{\Q}{\mathbb{Q}}
\newcommand{\Z}{\mathbb{Z}}
\newcommand{\N}{\mathbb{N}}
\newcommand{\Id}{\mathrm{Id}}

\section{Introduction}\label{Intro}

In this paper, we study rank-two \textit{stable wild Vafa-Witten bundles} on $\mathbb{P}^2$. Wild Vafa-Witten bundles on a K\"ahler surface $X$ are pairs $(E,\Phi)$ consisting of a holomorphic vector bundle $E$ on $X$ and a holomorphic bundle map $\Phi :E\to E\otimes L$, called a \textit{Higgs field}, for some fixed holomorphic line bundle $L\to X$. When $L=K_X$, the canonical bundle of $X$, these pairs are simply called \textit{Vafa-Witten bundles} and were introduced by Vafa and Witten while studying the $S$-duality conjecture for $\mathcal{N}=4$ supersymmetric Yang-Mills theory \cite{MR1305096}. Such objects extend the notion of Higgs bundles on Riemann surfaces to K\"ahler surfaces, the wild pairs corresponding to the generalized Higgs bundles studied by Bottacin \cite{MR1334607} and Markman \cite{MR1300764}. The terminology wild originated in \cite{MR3731155} where the authors considered $L=K_X(D)$ for some effective divisor $D$ on a Riemann surface $X$, allowing the Higgs field to have poles at the points of $D$, thus resulting in ``wild'' behaviour.

On the complex projective plane $\mathbb{P}^2$, non-trivial stable Vafa-Witten bundles do not exist; to be precise, when $L=K_X$ and $(E,\Phi)$ is stable, one can show that $\Phi=0$ \cite{MR4204283}. Nonetheless, non-trivial stable wild Vafa-Witten bundles do exist on $\mathbb{P}^2$ when $L=\mathcal{O}(d)$ for $d\geq0$ \cite{Moraru}. We analyze rank-two stable pairs $(E,\Phi)$ with trace-free Higgs field $\Phi :E\to E\otimes \mathcal{O}(d)$, $d >0$, and determine the dimension of certain components of the moduli space of stable wild Vafa-Witten pairs. Our approach parallels Rayan's study of co-Higgs bundles on $\mathbb{P}^1$ \cite{MR3158239} and $\mathbb{P}^2$ \cite{MR3285779}.

We begin by analyzing stable wild Vafa-Witten pairs $(E,\Phi)$ in the case where $E$ splits as a direct sum of two line bundles and $\Phi :E\to E\otimes \mathcal{O}(d)$, $d >0$. And for a fixed split bundle $E$, we compute the dimension of the space of Higgs fields on $E$ modulo conjugation; this corresponds to the dimension of the Zariski tangent space to the moduli space at the point $(E,\Phi)$. Schwarzenberger proved that any rank-two holomorphic vector bundle $E\to\mathbb{P}^2$ arises as the push-forward of a line bundle on a double cover $f: Y^\ell\to\mathbb{P}^2$ branched over a curve of degree $2\ell$ \cite{MR0137712}; these are known as \textit{Schwarzenberger bundles} of type $\ell$. We consider Schwarzenberger bundles of type $\ell=1$ and $\ell=2$.

For fixed $c_1,c_2\in\mathbb{Z}$ and $d \in \mathbb{Z}^{>0}$, we denote $\mathcal{M}_{c_1,c_2,d}^{\mathrm{VW}}$ the moduli space of rank-two stable wild Vafa-Witten bundles $(E,\Phi)$ on $\mathbb{P}^2$ with $c_1(E)=c_1H$, $c_2(E)=c_2H^2$, and $\Phi\in H^0(\mathbb{P}^2,\mathrm{End}_0E\otimes\mathcal{O}(d))$. Note that if $c_1=r+s-1$ and $c_2=\frac{1}{2}(r(r-1)+s(s-1))$ for some $s,r\in\mathbb{Z}$ with $s\geq r$, then $\mathcal{M}^{\mathrm{VW}}_{c_1,c_2,d}$ contains stable pairs $(E,\Phi)$ where $E$ is a Schwarzenberger bundle of type $\ell=1$. We consider moduli spaces with such Chern classes for $d\geq1$.

When $d=1$, we give a complete classification of stable wild Vafa-Witten bundles. Moreover, we have the following: 

\begin{thm}\label{Thm2}
For $r,s\in\mathbb{Z}$ with $s\geq r$, let $c_1=r+s-1$ and $c_2=\frac{1}{2}(r(r-1)+s(s-1))$. The moduli space $\mathcal{M}_{c_1,c_2,1}^{\mathrm{VW}}$ is smooth of complex dimension 6.    
\end{thm}

When $d>1$, we consider connected components of $\mathcal{M}_{c_1,c_2,d}^{\mathrm{VW}}$ containing pairs $(E,\Phi)$ where $E$ is a Schwarzenberger bundle of type $\ell=1$. We prove the following:

\begin{thm}\label{Thm1}
     Let $d>1$ and $r,s\in\mathbb{Z}$ with $s\geq r$. Let $c_1=r+s-1$ and $c_2=\frac{1}{2}(r(r-1)+s(s-1))$, and let $\mathcal{C}$ be a connected component of $\mathcal{M}_{c_1,c_2,d}^{\mathrm{VW}}$ containing stable pairs whose underlying bundle is a Schwarzenberger bundle of type $\ell=1$. The complex dimension of $\mathcal{C}$ is then $\frac{3}{2}d(d+3)$.

\end{thm}

The moduli space of Higgs bundles over Riemann surfaces admits a natural algebraic action of $\mathbb{C}^*$, induced by scaling the Higgs field. This action plays a central role in the computation of various enumerative invariants via localization techniques. An analogous $\mathbb{C}^*$-action is present on the moduli space of wild Vafa-Witten pairs and can likewise be exploited to compute invariants. Among these are the \textit{Tanaka-Thomas invariants}, introduced by Tanaka and Thomas in \cite{MR4158461}, defined as integrals of the virtual Euler class of the virtual normal bundle of the compact $\mathbb{C}^*$-fixed locus of the moduli space of stable wild Vafa-Witten bundles (see \cite{MR1305096,MR4158461,MR4194304,Chen}, for a detailed discussion). We end the paper by describing the $\C^*$-fixed locus of the moduli space of stable wild Vafa-Witten bundles on $\mathbb{P}^2$, following the analysis of Tanaka and Thomas; the computation of the Tanaka-Thomas invariants will appear in a future paper.

The paper is organized as follows. Section \ref{Pre} contains the necessary background for this paper. Section \ref{Sec3} is devoted to the study of stable wild Vafa-Witten bundles when the underlying bundle is split or a Schwarzenberger bundle for $\ell=1,2$. Section \ref{Sec4} contains the deformation theory of stable wild Vafa-Witten bundles and a proof of Theorems \ref{Thm2} and \ref{Thm1}. Lastly, Section \ref{Sec5} contains a description of the $\mathbb{C}^*$-fixed locus of stable wild Vafa-Witten bundles on $\mathbb{P}^2$.    

\section{Preliminaries}\label{Pre}

In this section, a brief overview of the basic objects and concepts needed in this paper are presented.

\subsection*{Definitions.}

We begin by introducing some definitions.
\begin{defn}
    Let $X$ be a complex manifold with canonical bundle $K_X$. A \textit{Vafa-Witten bundle} is a pair $(E,\Phi)$ where $E\to X$ is a holomorphic vector bundle and $\Phi:E\to E\otimes K_X$ is a holomorphic bundle map called a \textit{Higgs field}. Equivalently, $\Phi\in H^0(X,\mathrm{End}E\otimes K_X)$.
\end{defn}

These pairs were first considered by Vafa and Witten \cite{MR1305096} for $X$ a K\"ahler complex surface. In this paper, we consider the more general case where the canonical bundle $K_X$ is replaced by any holomorphic line bundle $L\to X$.

\begin{defn}
    Let $X$ be a complex manifold and $L\to X$ be a fixed holomorphic line bundle. A \textit{wild Vafa-Witten bundle} is a pair $(E,\Phi)$ where $E\to X$ is a holomorphic vector bundle and $\Phi:E\to E\otimes L$ is a holomorphic bundle map called a \textit{Higgs field}. Equivalently, $\Phi\in H^0(X,\mathrm{End}E\otimes L)$.
\end{defn}

This is a higher dimensional analogue of the generalized Higgs bundles considered by Bottacin \cite{MR1334607} and Markman \cite{MR1300764} on Riemann surfaces. Physicists are particularly interested in the case $L=K_X(D)$ for some effective divisor $D$ on $X$ (see for example \cite{MR3731155}). 

\begin{exmp}
    On $X=\mathbb{P}^2$, holomorphic line bundles are of the form $\mathcal{O}(d)$ for some $d\in\mathbb{Z}$. Thus, a wild Vafa-Witten pair $(E,\Phi)$ is such that $\Phi \in H^0(\mathbb{P}^2,\mathrm{End}E \otimes \mathcal{O}(d))$ for some $d\in\mathbb{Z}$. 
\end{exmp}

We are interested in moduli spaces of \textit{stable} wild Vafa-Witten bundles. In order to define stability, we need the notion of degree.

\begin{defn} 
Let $X$ be a complex projective manifold of dimension $n$ and $E\to X$ be a rank $r$ holomorphic vector bundle and choose any ample line bundle $H$ on $X$. We define the \textit{degree} of $E$ to be 

$$\deg_H(E):=c_1(E)\cdot H^{n-1}.$$ 
\end{defn}

\begin{exmp}
    On $X=\mathbb{P}^2$, $n=2$ and so $\deg_H(E)=c_1(E)\cdot H$. We choose the ample line bundle $H=\mathcal{O}(1)$ for the definition of the degree throughout the paper. For $E\to \mathbb{P}^2$ a rank $r$ holomorphic vector bundle, $c_1(E)=c_1(\mathcal{O}(m))$ for some $m\in\mathbb{Z}$. Therefore, $$\deg_H(E)=c_1(\mathcal{O}(m))\cdot c_1(\mathcal{O}(1))=m.$$
\end{exmp}

\begin{defn}
Let $X$ be a complex projective manifold with fixed ample line bundle $H$. Let $(E,\Phi)$ be a wild Vafa-Witten bundle over $X$ with $\Phi\in H^0(X,\mathrm{End}E\otimes L)$ for some line bundle $L$ on $X$. We say that the pair $(E,\Phi)$ is \textit{stable} (with respect to $H$) if for any non-zero proper subsheaf $F\subset E$ satisfying $\Phi(F)\subseteq F\otimes L$, we have that 

\begin{align}\label{eqn2.1}
    \frac{\mathrm{deg}_H(F)}{\mathrm{rank}(F)}<\frac{\mathrm{deg}_H(E)}{\mathrm{rank}(E)}.
\end{align}    

We call the pair \textit{semi-stable} if we allow for equality in the above inequality. 
\end{defn}

\begin{rem}
    Since we are only dealing with rank-two bundles $E$ over $\mathbb{P}^2$, it suffices to only check the inequality \eqref{eqn2.1} for sub-line bundles of $E$. See \cite[p.87]{MR1600388}.
\end{rem}

\begin{exmp}
    On $X=\mathbb{P}^2$, with the ample line bundle $H=\mathcal{O}(1)$, let $(E,\Phi)$ be a rank-two wild Vafa-Witten bundle over $\mathbb{P}^2$ with $\Phi\in H^0(\mathbb{P}^2,\mathrm{End}E\otimes \mathcal{O}(d))$, $d\geq0$. Then, the pair $(E,\Phi)$ is stable if for every sub-line bundle $L\subset E$ satisfying $\Phi(L)\subseteq L\otimes \mathcal{O}(d)$, we have $c_1(L)<\frac{1}{2}c_1(E)$.
\end{exmp}

\begin{rem}\label{rem4}
    It is worth noting that from the definition of stability of a vector bundle that a vector bundle $E$ is stable if and only if there exists a line bundle $L$ such that $E\otimes L$ is stable if and only if for all line bundles $L$, $E\otimes L$ is stable.
\end{rem}

\begin{rem}
    In this paper, we use the terminology stable pairs interchangeably with stable wild Vafa-Witten bundles.
\end{rem}

When classifying stable wild Vafa-Witten bundles, one does so up to isomorphism.

\begin{defn}
    Let $X$ be a complex projective manifold with fixed ample line bundle $H$. Given two stable wild Vafa-Witten bundles $(E_1,\Phi_1)$ and $(E_2,\Phi_2)$, we say that $(E_1,\Phi_1)$ is isomorphic to $(E_2,\Phi_2)$ if there exists an isomorphism of the underlying holomorphic bundles $f:E_1\to E_2$ such that $\Phi_1=f^{-1}\circ\Phi_2\circ f$.
\end{defn}

\begin{rem} 
In the above, when $E_1=E_2$, pairs $(E,\Phi_1)$ and $(E,\Phi_2)$ are isomorphic if there exists an automorphism $g:E\to E$ such that $\Phi_1=g^{-1}\circ\Phi\circ g$. In this case, pairs are isomorphic when the Higgs fields are equivalent up to conjugation by automorphisms of the bundle.
\end{rem}

\subsection*{The spectral correspondence.}
There is a correspondence relating stable Higgs bundles over Riemann surfaces to spectral covers called the \textit{spectral correspondence} that was first developed by Hitchin \cite{HitchinStableBundles}. This also holds for stable pairs $(E,\Phi)$ on complex projective surfaces $X$ with $\Phi\in H^0(X,\mathrm{End}E\otimes L)$ for $L$ a fixed line bundle on $X$. Let us briefly describe this correspondence.

Suppose $X$ is a complex surface with $L$ a fixed line bundle on $X$. Let $E\to X$ be a rank $r$ holomorphic vector bundle and $\Phi\in H^0(X,\mathrm{End}E\otimes L)$. Consider the pullback bundle $\pi^* L\to\mathrm{Tot}(L)$ over the total space of $L$. The characteristic polynomial of $\Phi$ over $\mathrm{Tot}(L)$ is 

\begin{align}\label{chareqn}
\mathrm{char}_{\eta}(\Phi)(p,\uplambda):=\eta^r(p,\uplambda)+(\eta^{r-1}\otimes\pi^*\phi)(p,\uplambda)+(\eta^{r-2}\otimes\pi^*\phi^2)(p,\uplambda)+\cdots+\pi^*\phi^r(p,\uplambda),
\end{align}

where $\eta:\mathrm{Tot}(L)\to\pi^*L$ is the tautological section of $\pi^*L$, defined by 

$$\eta(p,\uplambda):=((p,\uplambda),\uplambda))\in \pi^*L,$$ 

and 

$$\phi^i:=(-1)^i \mathrm{Tr} (\land^i\Phi)\in H^0(X,L^{\otimes i}).$$ If $\Phi$ has distinct eigenvalues, then the zero locus of \eqref{chareqn} defines a non-singular complex surface we denote $Y$. The map $\pi$ restricted to $Y$ gives an $r$:1 cover $Y\to X$ with branching locus given by the zeroes of \eqref{chareqn}. Given a line bundle $M$ on $Y$, the push-forward $\pi_*M$ is a rank $r$ holomorphic vector bundle on $X$. For open sets $U\subseteq X$, there is the mapping $M|_{\pi^{-1}(U)}\to M\otimes \pi^*L|_{\pi^{-1}(U)}$ defined by $s\mapsto s\otimes\eta$. The push-forward of this map gives a mapping

$$\pi_*M|_U\to\pi_*M\otimes L|_U,$$ where $\pi_*M\otimes L=\pi_*(M\otimes \pi^*L)$ using the projection formula. This map is a Higgs field $\Phi$ for the bundle $\pi_*M$ and the pair $(\Phi,\pi_*M)$ is stable. This implies that if one wants to study rank-two stable wild Vafa-Witten pairs $(E,\Phi)$ on $\mathbb{P}^2$, one needs to study double covers of $\mathbb{P}^2$. For a more in depth presentation on the spectral correspondence, see \cite{HitchinStableBundles,MR1334607,MR1300764,MR3884746,MR3837869,banerjee2025generalizedspectralcorrespondence} and the references therein.

\subsection*{Double covers of $\mathbb{P}^2$.}
The non-singular surfaces $Y$ appearing as double covers $Y\to \mathbb{P}^2$ in the spectral correspondence were studied by Schwarzenberger \cite{MR0137712} to construct rank-two holomorphic vector bundles over $\mathbb{P}^2$. His result is:

\begin{thm}[Schwarzenberger]
    Let $E\to\mathbb{P}^2$ be a rank-two holomorphic vector bundle. Then there exists a non-singular surface $Y$, a line bundle $L\to Y$, and a double cover $f:Y\to \mathbb{P}^2$ such that $E=f_*L$.
\end{thm}

Let $f:Y^\ell\to\mathbb{P}^2$ be a double cover of the projective plane branched along a smooth curve of positive degree $2\ell$. The geometry of $Y^\ell$ depends on the value of $\ell$. When $\ell=1$, one has $Y^1\cong \mathbb{P}^1\times\mathbb{P}^1$; for $\ell=2$, $Y^2$ is a del Pezzo surface of degree two, equivalently the blow-up of $\mathbb{P}^2$ at seven distinct points in general position; for $\ell=3$, $Y^3$ is a K3 surface; and for $\ell\geq4$, the surfaces $Y^\ell$ are of general type \cite{MR818312}. For any line bundle $L\to Y^\ell$, the push-forward $f_* L$ defines a rank-two vector bundle on $\mathbb{P}^2$. In what follows, we describe the rank-two vector bundles arising from $Y^1$ and $Y^2$. The case $\ell=3$ is only partially understood, and, to the best of the author's knowledge, no results are available for $\ell\geq4$.

Let $\ell=1$. In this case, the double cover $f:\mathbb{P}^1\times\mathbb{P}^1 \to \mathbb{P}^2$ is branched along a conic, i.e., a curve of degree two. The vector bundles of rank-two on $\mathbb{P}^2$ obtained as push-forwards of line bundles on $\mathbb{P}^1\times\mathbb{P}^1$ under $f$ admit an explicit description. Let $\pi_i:\mathbb{P}^1\times\mathbb{P}^1\to\mathbb{P}^1$ be the projection map onto the $i$-th factor for $i=1,2$. Every line bundle on $\mathbb{P}^1\times\mathbb{P}^1$ is of the form 

$$\mathcal{O}(r,s):=\pi_1^*\mathcal{O}(r)\otimes \pi_2^*\mathcal{O}(s),$$

for $r,s\in\mathbb{Z}$. For each such line bundle, we denote the associated rank-two bundle on $\mathbb{P}^2$ by 

$$E_{r,s}:=f_*\mathcal{O}(r,s).$$ 

Given any $\rho\in H^0(\mathbb{P}^2,\mathcal{O}(2))$, its zero locus determines a conic, and hence defines a double cover $f_\rho:\mathbb{P}^1\times\mathbb{P}^1\to\mathbb{P}^2$ branched along this conic. To record the choice of branching curve, we write $E_{r,s}^\rho:=f_{\rho,*}\mathcal{O}(r,s)$ for the resulting bundle. The bundles $E_{r,s}^\rho$ are known as \textit{Schwarzenberger bundles} of type $\ell=1$, originally introduced by Schwarzenberger in \cite{MR0137712}. Their basic properties are discussed in \cite{MR0137712,MR3285779}.

\begin{rem}[\textbf{Some Properties of the Bundles $E_{r,s}^\rho$}]
The bundles $E_{r,s}^\rho$ have the following properties:

\begin{enumerate}[i)]
    \item $E_{r,s}^\rho\cong E_{s,r}^\rho$. Because of this, without loss of generality, we will always assume $s\geq r$ in this paper.

    \item $E_{r,s}^\rho\cong E_{r',s'}^\rho$ if and only if $(r,s)=(r',s')$ or $(r,s)=(s',r')$.
    
    \item $c_1(E_{r,s}^\rho)=(r+s-1)\cdot H$, $c_2(E_{r,s}^\rho)=\frac{1}{2}(r(r-1)+s(s-1))\cdot H^2$ where $H$ is the hyperplane section on $\mathbb{P}^2$ (or, $H=c_1(\mathcal{O}(1))$ for $\mathcal{O}(1)$ on $\mathbb{P}^2$).
    
    \item Given two conics $\rho,\rho'\in H^0(\mathbb{P}^2,\mathcal{O}(2))$ as branching curves for the map $\mathbb{P}^1\times\mathbb{P}^1\to\mathbb{P}^2$, then $E_{r,s}^{\rho'}\cong E_{r,s}^\rho$ for $s=r,r+1,r+2$ and $E_{r,s}^{\rho'}\ncong E_{r,s}^\rho$ for $s=r+k$ with $k>2$.
     
    \item $E_{r,r}^{\rho}\cong \mathcal{O}(r)\oplus \mathcal{O}(r-1)$, $E_{r,r+1}^\rho\cong\mathcal{O}(r)\oplus\mathcal{O}(r)$, $E_{r,r+2}^\rho\cong T{\mathbb{P}^2}\otimes\mathcal{O}(r-1)$, and these isomorphisms are independent of the choice of branching curve $\rho\in H^0(\mathbb{P}^2,\mathcal{O}(2))$.
\end{enumerate}
\end{rem}

Let $\ell=2$. The double cover $g:Y^2\to\mathbb{P}^2$ is now branched over a curve of degree four; we denote the double cover in this case by $g$ to avoid confusion with the map $f$ when $\ell=1$. Referring to \cite{MR0137712}, the smooth surface $Y^2$ is the blow-up of $\mathbb{P}^2$ along seven distinct points $p_1,p_2,\dots,p_7\in\mathbb{P}^2$ in general position (that is, no three points are collinear and no six lie on a conic). As such, we denote $Y^2$ by $\tilde{\mathbb{P}}^2_7$. As for Schwarzenberger bundles of type $\ell=1$, we construct rank-two bundles on $\mathbb{P}^2$ by taking push-forwards of line bundles on $\tilde{\mathbb{P}}^2_7$ to $\mathbb{P}^2$. Line bundles on $\tilde{\mathbb{P}}^2_7$ have the following description \cite{MR0137712}. Denote by $N_1,N_2,\cdots,N_7$ the exceptional divisors in $\tilde{\mathbb{P}}^2_7$ obtained by blowing up the points $p_1,p_2,\dots,p_7$, respectively. Let $H$ be the hyperplane section on $\mathbb{P}^2$ and set $M:=g^*H$. Any divisor on $\tilde{\mathbb{P}}^2_7$ has the form $pM+\sum_{i=1}^7 t_iN_i$ with $p,t_1,t_2,\cdots,t_7\in\Z$. We denote 

$$L^{p,t_i}:=\mathcal{O}\left(pM+\sum_{i=1}^7 t_iN_i\right)$$ 

for any $p,t_1,t_2,\ldots,t_7\in\mathbb{Z}$. Let $g_\rho:\tilde{\mathbb{P}}^2_7\to \mathbb{P}^2$ be the double cover determined by $\rho\in H^0(\mathbb{P}^2,\mathcal{O}(4))$, which is branched along the quartic $\rho=0$, and let 

$$E_{p,t_i}^\rho:=g_{\rho,*}L^{p,t_i}$$ 

be the corresponding rank-two bundle on $\mathbb{P}^2$. We call these bundles Schwarzenberger bundles of type $\ell=2$. Their basic properties are discussed in \cite{MR0137712}.

\begin{rem}[\textbf{Some Properties of the Bundles $E^\rho_{p,t_i}$}]\label{prop2}

    Suppose $p,t_1,t_2,\dots,t_7\in\mathbb{Z}$. Then the bundles $E_{p,t_i}^\rho$ have the following properties:

\begin{enumerate}[i)]

\item If $E_{p,t_i}^\rho=E_{q,s_i}^\rho$ then $p=3\sum_is_i+8q$ and $t_i=-s_i-3q-\sum_js_j$.

    \item $$c_1(E_{p,t_i}^\rho)=\left(\sum_{t=1}^7t_i+3p-2\right)H.$$ \\

    \item $$c_2(E_{p,t_i}^\rho)=\left(4p^2-3p+(3p-1)\sum_{j=1}^7t_j+\sum_{j=1}^7t_j^2+\sum_{i<j}t_it_j\right)H^2.$$ \\ 

    \item  For $p\geq-1$, $$h^0(\mathbb{P}^2,E_{p,t_i}^\rho)=\begin{cases} \frac{1}{2}(p+1)(p+2) & \text{if all $t_i\geq0$} \\ d(p,t_i) & \text{some $t_i<0$} \end{cases}.$$ \\ 

    \item For $p\geq-1$, $$h^1(\mathbb{P}^2,E_{p,t_i}^\rho)=\begin{cases} \frac{1}{2}\sum_{j=1}^7t_j(t_j-1) & \text{if all $t_i\geq0$} \\ d(p,t_i)-\frac{1}{2}(p+1)(p+2)+\frac{1}{2}
    \sum_{j=1}^7t_j(t_j-1)& \text{some $t_i<0$} \end{cases}$$ where, $$d(p,t_i)=\frac{1}{2}(p+1)(p+2)-\frac{1}{2}\sum_{j=1}^7t_j(t_j+1)$$ is the dimension of the vector space of curves of degree $p$ having multiplicity $-t_i$ at the points $p_i$ for which $t_i<0$. If $p<-1$ in properties iv) and v), then one simply uses Serre duality on $h^r(\mathbb{P}^2,E_{p,t_i}^\rho)$ for $r=1,2$. See \cite[Proposition 13]{MR0137712}.

    \end{enumerate}
\end{rem}

Note that certain rank-two holomorphic bundles over $\mathbb{P}^2$ are Schwarzenberger bundles of type both $\ell=1$ and $\ell=2$. For example, Schwarzenberger gives a classification of non-homogeneous rank-two bundles that are Schwarzenberger bundles of types both of $\ell=1$ and $\ell=2$. Recall that a rank-two vector bundle over $\mathbb{P}^2$ is said to be \textit{homogeneous} if it is of the form $\mathcal{O}(a)\oplus\mathcal{O}(b)$, $T{\mathbb{P}^2}\otimes\mathcal{O}(c)$, or $T^*{\mathbb{P}^2}\otimes\mathcal{O}(d)$ for some $a,b,c,d\in\mathbb{Z}$. We have:

\begin{prop}(\cite[Proposition 11]{MR0137712})\label{Prop9}
    Let $E$ be a non-homogeneous rank-two vector bundle on $\mathbb{P}^2$. Suppose that $E=E_{r,s}^\rho$ for some $r,s\in\mathbb{Z}$ and $\rho\in H^0(\mathbb{P}^2,\mathcal{O}(2))$, and that $E=E_{p,t_i}^{\rho'}$ for some $p,t_1,t_2,...,t_7\in\mathbb{Z}$ and $\rho'\in H^0(\mathbb{P}^2,\mathcal{O}(4))$. Then, $p\in\{-7,-6,...,-1\}$, $r=p$, $s=p+3$, and $\sum_{i=1}^7t_i=4-p$. In particular, $E$ is stable.
\end{prop}

\subsection*{$\mathbb{C}^*$-locus.}
There is a natural $\mathbb{C}^*$ group action on Higgs bundles over curves, used to study the topology of the moduli space of stable Higgs bundles (see for example \cite{MR887284,MR3884746}). This action also exists in the case of wild Vafa-Witten bundles over $\mathbb{P}^2$. It is defined as follows: given a stable wild Vafa-Witten bundle $(E,\Phi)$ and $c\in\mathbb{C}^*$, we set $c\cdot(E,\Phi)=(E,c\Phi)$. The $\mathbb{C}^*$-fixed points have the following structure:

\begin{thm}\label{Thm8}
    Let $d\geq0$. Given a $\mathbb{C}^*$-invariant stable wild Vafa-Witten bundle $(E,\Phi)$ with $\Phi\in H^0(\mathbb{P}^2,\mathrm{End}E\otimes\mathcal{O}(d))$, the Higgs field $\Phi$ is nilpotent, i.e., $\Phi^n=0$ for some $n \in \mathbb{N}$. Furthermore, the bundle $E$ decomposes as

    $$E=\bigoplus_{i,j}E_{i,j},$$ where $E_{i,j}$ are sub-bundles of $E$ and $\Phi:E_{i,j}\to E_{i-1,j+1}(d)$.
\end{thm}

For more details, see \cite[Section 7.3]{MR4158461} and \cite[Proposition 2.15]{Chen}. One can use virtual cycles in the sense of Behrend-Fantechi \cite{MR1437495} on the $\mathbb{C}^*$-fixed locus to compute the Tanaka-Thomas invariants of the moduli space \cite{MR4158461}.

\subsection*{Notation.}
We denote $E(d):=E\otimes \mathcal{O}(d)$ throughout for any holomorphic vector bundle $E$ on $\mathbb{P}^2$ and $d\in\mathbb{Z}$.

\section{Stable Wild Vafa-Witten Bundles}\label{Sec3}

In this section, we consider rank-two stable wild Vafa-Witten bundles $(E,\Phi)$ with $E$ a split bundle or a Schwarzenberger bundle of type $\ell=1$ or $\ell=2$, and $\Phi\in H^0(\mathbb{P}^2,\mathrm{End}E(d))$, $d\geq0$. When $E=\mathcal{O}(m_1)\oplus \mathcal{O}(m_2)$ with $m_1,m_2\in\mathbb{Z}$, by deformation theory, the Zariski tangent space to the moduli space $\mathcal{M}^{\mathrm{VW}}_{m_1+m_2,m_1m_2,d}$ at the point $(E,\Phi)$ corresponds to the space of trace-free Higgs fields $\Phi'\in H^0(\mathbb{P}^2,\mathrm{End}_0E(d))$ with $(E,\Phi')$ stable, up to conjugation by automorphisms of $E$ (see Section \ref{Sec4}); we compute the dimension of this space. When the underlying bundle is a Schwarzenberger bundle of type $\ell=1$, we compute the dimensions $h^i(\mathbb{P}^2,\mathrm{End}_0E(d))$, $i>0$, which appear in the deformation theory of wild Vafa-Witten pairs (see Section \ref{Sec4}). Finally, we determine which Schwarzenberger bundles of type $\ell=2$ are stable; for such bundles, the pair $(E,\Phi)$ is stable for any Higgs field $\Phi$.

\subsection{Split Case}\label{split}

We first consider pairs $(E,\Phi)$ with $E$ the sum of two line bundles and $\Phi\in H^0(\mathbb{P}^2,\mathrm{End}E(d))$, $d\geq0$. Let $E=\mathcal{O}(m_1)\oplus \mathcal{O}(m_2)$ with $m_1,m_2\in\mathbb{Z}$. In this case, the Higgs field may be express globally as 

$$\Phi=\begin{pmatrix} A & B \\ C & D \end{pmatrix}$$ with $A\in H^0(\mathbb{P}^2,\mathcal{O}(d))$, $B\in H^0(\mathbb{P}^2,\mathcal{O}(m_1-m_2+d))$, $C\in H^0(\mathbb{P}^2,\mathcal{O}(m_2-m_1+d))$, and $D\in H^0(\mathbb{P}^2,\mathcal{O}(d))$. As in \cite[Proposition 5.1]{MR3285779}, stability imposes conditions on $m_1$ and $m_2$:

\begin{prop}\label{2.1}
    Let $d\geq0$ and $E=\mathcal{O}(m_1)\oplus \mathcal{O}(m_2)$ with $m_1,m_2\in\mathbb{Z}$. Suppose that there exists a $\Phi\in H^0(\mathbb{P}^2,\mathrm{End}E(d))$ for which the pair $(E,\Phi)$ is stable. Then $|m_1-m_2|\leq d$. 
\end{prop}

\begin{proof}
Assume without loss of generality that $m_1\geq m_2$. If $m_1-m_2>d$, then $C=0$ and $\mathcal{O}(m_1)$ becomes $\Phi$-invariant and destabilizing, contradicting the stability of $(E,\Phi)$. 
\end{proof}

For fixed $E$, let us now compute the dimension of the space of Higgs fields $\Phi$ on $E$ with $(E,\Phi)$ stable, up to conjugation by automorphisms of $E$. Up to tensoring by a line bundle, we can assume that $E=\mathcal{O}\oplus\mathcal{O}(m)$ with $m \geq 0$ (see Remark \ref{rem4}).

\begin{prop}\label{Prop10}
    Let $d\geq 0$ and $0\leq m\leq d$. Let $E=\mathcal{O}\oplus\mathcal{O}(m)$. When $m=0$, the dimension of the space of Higgs fields $\Phi\in H^0(\mathbb{P}^2,\mathrm{End}E(d))$ with $(E,\Phi)$ stable, up to conjugation by automorphisms of $E$, is $2d^2+6d+1$. For $m\neq0$, the dimension is $2d^2+6d+\frac{1}{2}m(m-3)+2$.
\end{prop}

\begin{rem}
    If we let $m\leq0$ in the proposition, a similar computation gives the same dimension. However, stability imposes $|m|\leq d$ in this case.
\end{rem}

\begin{proof} 

We can represent $\Phi$ globally as a $2\times 2$ matrix $$\Phi=\begin{pmatrix} A & B \\ C & D \end{pmatrix},$$ where $A\in H^0(\mathcal{O}(d))$, $B\in H^0(\mathcal{O}(d-m))$, $C\in H^0(\mathcal{O}(d+m))$, and $D\in H^0(\mathcal{O}(d))$. The stability of the pair $(E,\Phi)$ in this case requires $m\leq d$. Counting how many degrees of freedom $\Phi$ has amounts to adding the dimensions of $H^0(\mathbb{P}^2,\mathcal{O}(i))$ for $i=d,d+m,d-m$. This gives

$$\sum_{i\in\{d-m,d,d+m\}}h^0(\mathbb{P}^2,\mathcal{O}(i))=2d^2+6d+m^2+4$$ 

degrees of freedom for $\Phi$. Subtracting by the action of the automorphism group of $E=\mathcal{O}\oplus\mathcal{O}(m)$ counts Higgs fields up to conjugation. If $\Psi$ is an automorphism of $E$, it can be written as 

$$\Psi=\begin{pmatrix} a & b \\ c & d \end{pmatrix}$$ with $a,d\in H^0(\mathbb{P}^2,\mathcal{O})=\mathbb{C}$, $b\in H^0(\mathbb{P}^2,\mathcal{O}(-m))$, and $c\in H^0(\mathbb{P}^2,\mathcal{O}(m))$. Also, conjugation of the Higgs field $\Phi$ is invariant up to scaling by a non-zero constant. I.e. if $c\in\mathbb{C}^*$, then 

$$(c\Psi)\cdot\Phi\cdot(c\Psi)^{-1}=\Psi\cdot\Phi\cdot\Psi^{-1}$$ and conjugation by $\Psi$ is unique up to a choice of an element in $\mathbb{C}^*$.

When $m>0$, then $b=0$ and there are two choices for $a,d$ and $\frac{(m+2)(m+1)}{2}$ choices for $c$. This gives a total of $2+\frac{(m+2)(m+1)}{2}$ degrees of freedom for a choice of $\Psi$. So in order to compute all degrees of freedom for Higgs fields up to conjugation, we must subtract $2d^2+6d+m^2+4$ by $1+\frac{(m+2)(m+1)}{2}$ giving the desired result.

Suppose now $m=0$, we have $a,b,c,d\in H^0(\mathbb{P}^2,\mathcal{O})=\mathbb{C}$. The correction needed to account for conjugation of the Higgs field is $4-1=3$. Thus the total number of Higgs fields for which the pair $(E,\Phi)$ is stable up to conjugation is $2d^2+6d+4-3=2d^2+6d+1$ as required.

\end{proof}

In deformation theory, one considers trace-free Higgs fields. Note that $$H^0(\mathbb{P}^2,\mathrm{End}E(d))=H^0(\mathbb{P}^2,\mathrm{End}_0E(d))\oplus H^0(\mathbb{P}^2,\mathcal{O}(d)),$$ where $\mathrm{End}_0$ denotes trace-free endomorphisms. Thus to get the total number of trace-free Higgs fields such that the pair $(E,\Phi)$ is stable, we must subtract $h^0(\mathbb{P}^2,\mathcal{O}(d))=\frac{1}{2}(d+1)(d+2)$ to the dimensions computed in Proposition \ref{Prop10}. In this case, these represent the dimension of the Zariski tangent space to $\mathcal{M}^{\mathrm{VW}}_{m,0,d}$ at the point $(E,\Phi)$. This yields

\begin{cor} 
Let $d\geq0$ and $0\leq m \leq d$ and let $E=\mathcal{O}\oplus\mathcal{O}(m)$. The dimension of the Zariski tangent space to $\mathcal{M}^{\mathrm{VW}}_{m,0,d}$ at the point $(E,\Phi)$ is $\frac{3}{2}d(d+3)$ when $m=0$ and $\frac{3}{2}d(d+3)+\frac{1}{2}m(m-3)+1$ when $m>0$. 
\end{cor}

From now on, we assume that the Higgs field $\Phi$ is trace-free.

\subsection{Wild Vafa-Witten Pairs With $\ell=1$}\label{3.2}

Let $(E,\Phi)$ be a rank-two wild Vafa-Witten bundle with $\Phi\in H^0(\mathbb{P}^2,\mathrm{End}_0E(1))$. In this case, $\det \Phi \in H^0(\mathbb{P}^2,\mathcal{O}(2))$ and the spectral cover of $\Phi$ is a double cover $f:Y\to \mathbb{P}^2$ branched along the conic $\rho=\det \Phi$. By the spectral correspondence, if the conic is smooth, then $(E,\Phi)$ is stable, and there exists a line bundle $L\to Y$ such that $E=f_* L$ and $\Phi$ is determined by the multiplication map induced by the tautological section on $\mathrm{Tot}(L)$. Since $f$ is a double cover of $\mathbb{P}^2$ branched along a smooth conic, Schwarzenberger's classification implies that $Y\cong\mathbb{P}^1\times\mathbb{P}^1$ \cite{MR0137712}. The push-forwards of line bundles on $\mathbb{P}^1\times\mathbb{P}^1$ are explicitly known; they coincide with the bundles $E_{r,s}^\rho$ introduced in Section \ref{Pre}. 
In this section, we compute the dimensions $h^i(\mathbb{P}^2,\mathrm{End}_0E_{r,s}^\rho(d))$ for $i\geq0$ and $d\geq0$, which are necessary for the deformation theory of pairs $(E_{r,s}^\rho,\Phi)$ (see Section \ref{Sec4}). The case where $s=0$ and $r\geq3$ is treated in \cite[Proposition 6.1 and Corollary 6.3]{MR3285779} and \cite[Corollary 5.4]{SSR:11}; the proof for the remaining cases proceeds analogously and is included here for completeness.

\begin{prop}\label{Prop11}
    Let $d\geq0$ and assume without loss of generality that $s\geq r$. The dimension of the zeroth cohomology group for $\mathrm{End}_0E_{r,s}^\rho(d)$ is given by 

$$h^0(\mathbb{P}^2,\mathrm{End}_0E_{r,s}^\rho(d))=\frac{1}{2}d(d+1)+\delta_{rs,d}(2-s+r+d)(2-r+s+d),$$ where $\delta_{rs,d}=1$ if $d\geq s-r-1$ and zero otherwise.
\end{prop}

\begin{proof} 

Let $d>0$ and notice that 

\begin{align*}
H^0(\mathbb{P}^2,\mathrm{End}E_{r,s}^\rho(d))&= H^0(\mathbb{P}^2,E_{r,s}^\rho\otimes (E_{r,s}^\rho)^*\otimes\mathcal{O}(d)) \\
&=H^0(\mathbb{P}^1\times\mathbb{P}^1,\mathcal{O}(r,s)\otimes f^*((E_{r,s}^\rho)^*\otimes\mathcal{O}(d))) \\
&=H^0(\mathbb{P}^1\times\mathbb{P}^1,f^*(E_{r,s}^\rho)^*\otimes \mathcal{O}(r,s)\otimes\mathcal{O}(d,d)) \\ 
&=H^0(\mathbb{P}^1\times\mathbb{P}^1,f^*(E_{r,s}^\rho)^*(r+d,s+d))
\end{align*}

using the projection formula with $f:\mathbb{P}^1\times\mathbb{P}^1\to\mathbb{P}^2$ and the fact that $E_{r,s}^\rho=f_*\mathcal{O}(r,s)$ and $f^*\mathcal{O}(d) = \mathcal{O}(d,d)$. So we compute the dimension of the group $$H^0(\mathbb{P}^1\times\mathbb{P}^1,f^*(E_{r,s}^\rho)^*(r+d,s+d))$$ as it is easier to compute cohomology on $\mathbb{P}^1$. The map $f^*E_{r,s}^\rho\to \mathcal{O}(r,s)$ is surjective. i.e. it has a kernel, so there is the short exact sequence 

$$0\to K\to f^*E_{r,s}^\rho\to \mathcal{O}(r,s)\to 0$$ where $K$ is the kernel. Since $K$ is a line bundle on $\mathbb{P}^1\times\mathbb{P}^1$, it takes the form $K=\mathcal{O}(a,b)$ for some $a,b\in\mathbb{Z}$. Then, $\det(f^*E_{r,s}^\rho)= K \otimes \mathcal{O}(r,s) = \mathcal{O}(a+r,b+s)$. 
Since
$$\det(f^*E_{r,s}^\rho)\cong\mathcal{O}(c_1(E_{r,s}^\rho),c_1(E_{r,s}^\rho))=\mathcal{O}(r+s-1,r+s-1),$$ this implies $a=s-1$ and $b=r-1$. So, one obtains the short exact sequence 

$$0\to\mathcal{O}(s-1,r-1)\to f^*E_{r,s}^\rho\to\mathcal{O}(r,s)\to 0.$$ Taking the dual sequence and tensoring by $\mathcal{O}(r+d,s+d)$ gives

$$0\to\mathcal{O}(d,d)\to f^*(E_{r,s}^\rho)^*(r+d,s+d)\to\mathcal{O}(1-s+r+d,1-r+s+d)\to0.$$ This induces the long exact sequence of cohomologies 

\begin{align*}
0\to H^0(\mathcal{O}(d,d))\to H^0(f^*(E_{r,s}^\rho)^*(r+d,s+d)) \to H^0(\mathcal{O}(1-s+r+d,1-r+s+d)) &\to \\ \to H^1(\mathcal{O}(d,d))\to\cdots.
\end{align*}

But $H^1(\mathcal{O}(d,d))=0$ on $\mathbb{P}^1\times\mathbb{P}^1$, so it becomes the short exact sequence

$$0\to H^0(\mathcal{O}(d,d))\to H^0(f^*(E_{r,s}^\rho)^*(r+d,s+d))\to H^0(\mathcal{O}(1-s+r+d,1-r+s+d))\to0.$$ From this,

\begin{align*}
h^0(f^*(E_{r,s}^\rho)^*(r+d,s+d)) &= h^0(\mathcal{O}(d,d))+h^0(\mathcal{O}(1-s+r+d,1-r+s+d)) \\
&=(d+1)^2+\delta_{rs,d}(2-s+r+d)(2-r+s+d),
\end{align*}

where $\delta_{rs,d}=1$ if $d\geq s-r-1$ and zero otherwise. Subtracting by the trace gives 

$$h^0(\mathbb{P}^2,\mathrm{End}_0E_{r,s}^\rho(d))=\frac{1}{2}d(d+1)+\delta_{rs,d}(2-s+r+d)(2-r+s+d).$$

\end{proof}

Using these calculations, we may determine which of the $E_{r,s}^\rho$ are stable. $E_{r,s}^\rho$ is simple if $$h^0(\mathbb{P}^2,\mathrm{End}_0E_{r,s}^\rho)=0,$$ and rank-two bundles on $\mathbb{P}^2$ are simple if and only if they are stable. Thus, by setting $d=0$ in the above proposition, we get 

$$h^0(\mathbb{P}^2,\mathrm{End}_0E_{r,s}^\rho)=(2-s+r)(2-r+s)$$ for $s\leq r+1$ and $h^0(\mathbb{P}^2,\mathrm{End}_0E_{r,s}^\rho)=1$ for $s> r+1$. So, $E_{r,s}^\rho$ is simple and hence stable if and only if $s\geq r+2$.

\begin{prop}\label{Prop15}
    Let $d\geq0$, $r\in\mathbb{Z}$ and assume without loss of generality that $s\geq r$. Then $H^2(\mathbb{P}^2,\mathrm{End}_0E_{r,s}^\rho(d))=0$.
\end{prop}

\begin{proof}
    Use Serre duality on $H^2(\mathbb{P}^2,\mathrm{End}_0E_{r,s}^\rho(d))^*$ to get 

$$H^2(\mathbb{P}^2,\mathrm{End}_0E_{r,s}^\rho(d))^*\cong H^0(\mathbb{P}^2,\mathrm{End}_0E_{r,s}^\rho(-d-3))=H^0(\mathbb{P}^2,\mathrm{End}_0E_{r,s}^\rho(\tilde{d}))$$ where $\tilde{d}:=-d-3<0$. Plugging in the cases $s=r,r+1,r+2$ and $s=r+k$ for $k>2$ yields $h^0(\mathbb{P}^2,\mathrm{End}_0E_{r,s}^\rho(\tilde{d}))=0$.
\end{proof}

In the remainder, we suppress $\mathbb{P}^2$ from the cohomology groups to simplify notation.  

\begin{prop}\label{3.4}
    Let $d\geq0$. For $s=r$, $s=r+1$, or $s=r+2$ we have $h^1(\mathrm{End}_0E_{r,s}^\rho(d))=0$.
\end{prop}

\begin{proof}
    Let $d\geq0$. By Hirzeburch-Riemann-Roch, 

    $$h^1(\mathrm{End}_0E_{r,s}^\rho(d))=h^0(\mathrm{End}_0E_{r,s}^\rho(d))+h^2(\mathrm{End}_0E_{r,s}^\rho(d))-\deg(\mathrm{ch}(\mathrm{End}_0E_{r,s}^\rho(d)) \cdot\mathrm{td}(\mathbb{P}^2))_2.$$
We have
$$\mathrm{ch}(\mathrm{End}_0E_{r,s}^\rho)=3+(c_1(E_{r,s}^\rho)^2-4c_2(E_{r,s}^\rho))$$ with $c_1(E_{r,s}^\rho)=(r+s-1)\cdot H$ and $c_2(E_{r,s}^\rho)=\frac{1}{2}(r(r-1)+s(s-1))\cdot H^2$,

$$\mathrm{ch}(\mathcal{O}(d))=1+dH+\frac{1}{2}d^2H^2,$$
and
$$\mathrm{td}(\mathbb{P}^2)=1+\frac{3}{2}H+H^2.$$ 
Taking the product of the above and extracting the coefficient of $H^2$ yields

$$h^1(\mathrm{End}_0E_{r,s}^\rho(d)) = h^0(\mathrm{End}_0E_{r,s}^\rho(d))+h^2(\mathrm{End}_0E_{r,s}^\rho(d))-\left(\frac{3}{2}d(d+3)-r^2-s^2+2rs+4\right).$$ 

The result then follows from Propositions \ref{Prop11} and \ref{Prop15}.

\end{proof}

We end by considering the case $s>r+2$.

\begin{prop}\label{3.5}
    Let $d\geq0$, $r\in\mathbb{Z}$ and assume $s=r+k$ for $k>2$. Then the following hold:
    \begin{enumerate}[i)]
        \item $h^1(\mathrm{End}_0E_{r,s}^\rho)=k^2-4$;
        \item $h^1(\mathrm{End}_0E_{r,s}^\rho(1))=k^2-9$;
        \item $h^1(\mathrm{End}_0E_{r,s}^\rho(d))=0$ for $d>1$.
    \end{enumerate}
\end{prop}

\begin{proof}
    From Proposition \ref{3.4}, we have the equation $$h^1(\mathrm{End}_0E_{r,r+k}^\rho(d))=-\frac{3}{2}d(d+3)+r^2+(r+k)^2-2r(r+k)-4+h^0(\mathrm{End}_0E_{r,r+k}^\rho(d)).$$ Combined with Proposition \ref{Prop11}, this yields the desired result. 
\end{proof}

\subsection{Wild Vafa-Witten Pairs With $\ell=2$}\label{3.3}

Let $(E,\Phi)$ be a rank-two wild Vafa-Witten bundle with $\Phi\in H^0(\mathbb{P}^2,\mathrm{End}_0E(2))$. This time, $\det \Phi \in H^0(\mathbb{P}^2,\mathcal{O}(4))$ and the spectral cover of $\Phi$ is a double cover $g:Y\to \mathbb{P}^2$ branched along the quartic $\rho=\det\Phi$. Again, the spectral correspondence tells us that if the quartic is smooth, then $(E,\Phi)$ is stable and there exists a line bundle $L\to Y$ such that $E=g_* L$ and $\Phi$ is given by the multiplication map on $\mathrm{Tot}(L)$. And since $g$ is a double cover of $\mathbb{P}^2$ branched along a smooth quartic, Schwarzenberger's classification implies that $Y$ is the blow-up $\tilde{\mathbb{P}}^2_7$ of $\mathbb{P}^2$ along seven distinct points in general position \cite{MR0137712}. The push-forwards of the line bundles $L^{p,t_i}$ on $\tilde{\mathbb{P}}^2_7$ are the bundles $E_{p,t_i}^\rho$ introduced in Section \ref{Pre}. In this case, we are not able to compute the dimensions $h^i(\mathbb{P}^2,\mathrm{End}_0E_{p,t_i}^\rho(d))$ for $i\geq0$ and $d\geq0$. In the case $\ell=1$, the calculation of $h^0(\mathbb{P}^2,\mathrm{End}_0E_{r,s}^\rho(d))$ in Proposition \ref{Prop11} is greatly simplified by the fact that $h^1(\mathbb{P}^1\times\mathbb{P}^1,\mathcal{O}(d,d))=0$, for all $d\geq0$. For the $\ell=2$ case, an analogous simplification does not occur because $h^1(\tilde{\mathbb{P}}^2_7,L^{3d,-d})=7d$, which is non-zero when $d>0$. We can nonetheless determine which of the bundles $E_{p,t_i}^\rho$ are stable. Note that when $E_{p,t_i}^\rho$ is stable, then $(E_{p,t_i}^\rho,\Phi)$ is a stable wild Vafa-Witten pair for any $d\geq0$ and any $\Phi\in H^0(\mathbb{P}^2,\mathrm{End}_0E_{p,t_i}^\rho(d))$. 

\begin{prop} \label{6.2}

    Suppose $p\geq-1$ and $t_1,t_2,\dots,t_7\in\mathbb{Z}$, then the bundles $E_{p,t_i}^\rho$ are stable if and only if 

\begin{align*}
\sum_{j=1}^7t_j=\frac{8-7p}{3} \,\ \text{or} \,\ \frac{7-7p}{3}&, \,\ \text{if $c_1+2t_i\geq0$ for all $i$}
\end{align*}
or
\begin{align*}
\sum_{j=1}^7 [(c_1+2t_j)^2+(c_1+2t_j)]=(-3c_1+2p+1)(-3c_1+2p+2), \,\ \text{if $c_1+2t_i<0$ for some $i$,}
\end{align*}
where $c_1=\sum_{i=1}^7 t_i+3p-2$.
    
\end{prop}

\begin{proof}
    From the surjective map $g^*E_{p,t_i}^\rho\to L^{p,t_i}$, one obtains the following short exact sequence 

    $$0\to K\to g^*E_{p,t_i}^\rho\to L^{p,t_i}\to0$$ where $K$ is the kernel bundle, which we may write as $K=L^{q,s_i}$ for some $q,s_1,s_2,\cdots,s_7\in\Z$. From $\det(g^*E_{p,t_i}^\rho)=L^{p,t_i}\otimes L^{q,s_i}=L^{p+q,t_i+s_i}$ and 

    $$\det(g^*E_{p,t_i}^\rho)=g^*\det(E_{p,t_i}^\rho)=g^*\mathcal{O}((t+3p-2)\cdot H)=K_{\tilde{\mathbb{P}}^2_7}^{-c_1}$$ we get that $L^{p+q,s_i+t_i}=K_{\tilde{\mathbb{P}}^2_7}^{-c_1}$ where $K_{\tilde{\mathbb{P}}^2_7}$ is the canonical bundle of $\tilde{\mathbb{P}}^2_7$. The canonical bundle $K_{\tilde{\mathbb{P}}^2_7}$ may be associated to the divisor $-3M+N_1+\cdots+N_7$, and we may write $K_{\tilde{\mathbb{P}}^2_7}=L^{-3,1}$ and so $K_{\tilde{\mathbb{P}}^2_7}^{-c_1}=L^{3c_1,-c_1}$. Thus,

    $$L^{3c_1,-c_1}=L^{p+q,t_i+s_i}$$ which implies $3c_1=p+q$ and $-c_1=t_i+s_i$ and so we have the short exact sequence

    $$0\to L^{3c_1-p,-c_1-t_i}\to g^*E_{p,t_i}^\rho\to L^{p,t_i}\to0.$$ Taking the dual sequence and then tensoring by $L^{p,t_i}\otimes K_{\tilde{\mathbb{P}}^2_7}^{-d}$ yields 

    $$0\to L^{3d,-d}\to L^{p,t_i}\otimes g^*((E_{p,t_i}^\rho)^*\otimes\mathcal{O}(d))\to L^{-3c_1+2p+3d,c_1+2t_i-d}\to0.$$ Taking the long exact sequence in cohomology gives

    \begin{align*}
        0\to H^0(\tilde{\mathbb{P}}^2_7,L^{3d,-d})\to H^0(\tilde{\mathbb{P}}^2_7,L^{p,t_i}\otimes g^*((E_{p,t_i}^\rho)^*(d))) &\to H^0(\tilde{\mathbb{P}}^2_7,L^{-3c_1+2p+3d,c_1+2t_i-d}) \\
        &\to H^1(\tilde{\mathbb{P}}^2_7,L^{3d,-d})\to\cdots
    \end{align*}    
        
        where using the properties from \ref{prop2} iv) and v), $h^1(\tilde{\mathbb{P}}^2_7,L^{3d,-d})=7d$ and $h^0(\tilde{\mathbb{P}}^2_7,L^{3d,-d})=d^2+8d+1$. So for $d=0$: 

    $$h^0(\mathbb{P}^2,\mathrm{End}E_{p,t_i}^\rho)=h^0(\tilde{\mathbb{P}}^2_7,L^{p,t_i}\otimes g^*(E_{p,t_i}^\rho)^*)=1+h^0(\tilde{\mathbb{P}}^2_7,L^{-3c_1+2p,c_1+2t_i})$$ and so the result follows from \ref{prop2} iv) and the fact that $$h^0(\tilde{\mathbb{P}}^2_7,L^{p,t_i}\otimes g^*(E_{p,t_i}^\rho)^*\otimes\mathcal{O}(d)))=h^0(\mathbb{P}^2,\mathrm{End}E_{p,t_i}^\rho(d)).$$ 
\end{proof}

\subsection{Relationship Between \texorpdfstring{$\ell=1$}{l=1} and \texorpdfstring{$\ell=2$}{l=2}}

Let $E\to\mathbb{P}^2$ be a rank-two holomorphic vector bundle arising both as $E=f_*\mathcal{O}(r,s)$ for a double cover $f:\mathbb{P}^1\times\mathbb{P}^1\to\mathbb{P}^2$ when $\ell=1$, and as $E=g_*L^{p,t_i}$ for a double cover $g:\tilde{\mathbb{P}}^2_7\to\mathbb{P}^2$ when $\ell=2$, as described in Sections \ref{3.2} and \ref{3.3}, respectively. As noted at the end of Remark \ref{prop2}, such bundles exist. Since both coverings are smooth, the spectral correspondence yields stable pairs in each case. When $\ell=1$, the correspondence produces a Higgs field $\Phi\in H^0(\mathbb{P}^2,\mathrm{End}_0E(1))$, and for $\ell=2$, it produces a Higgs field $\tilde{\Phi}\in H^0(\mathbb{P}^2,\mathrm{End}_0E(2))$. However, given any $h \in H^0(\mathbb{P}^2,\mathcal{O}(1))$, we can construct the Higgs field $\Phi\otimes h \in H^0(\mathbb{P}^2,\mathrm{End}_0E(2))$, and the pair $(E,\Phi\otimes h)$ defines a stable wild Vafa-Witten bundle with $d=2$ whose Higgs field is a multiple of a Higgs field with $d=1$. Therefore, when the underlying bundle $E$ is a Schwarzenberger bundle for both $\ell=1$ and $\ell=2$, Higgs fields with $d=2$ can be obtained via the spectral correspondence for $\ell=1$ after tensoring by an $h\in H^0(\mathbb{P}^2,\mathcal{O}(1))$, or via the spectral correspondence for $\ell=2$. In both cases, the wild Vafa-Witten pair is stable. However, in the first case, if $\Phi\in H^0(\mathbb{P}^2,\mathrm{End}_0E(1))$ and $h\in H^0(\mathbb{P}^2,\mathcal{O}(1))$, the spectral cover of $\Phi\otimes h$ is singular since its branching curve is the zero locus of $h^2\det\Phi$ and therefore reducible. Whereas, the Higgs fields coming from the $\ell=2$ construction have smooth spectral covers.

\subsection{Wild Vafa-Witten Bundles With $\ell\geq3$ or Singular Spectral Covers With $\ell\geq1$}
    When $\ell\geq3$, consider a smooth double cover $f:X\to\mathbb{P}^2$ branched along a smooth curve $\rho\in H^0(\mathbb{P}^2,\mathcal{O}(2\ell))$. Suppose $(E,\Phi)$ is a rank-two Vafa-Witten bundle with $\Phi\in H^0(\mathbb{P}^2,\mathrm{End}_0E(\ell))$ and $\det\Phi=\rho$. Then the spectral correspondence states that $E$ is obtained as the push-forward of some line bundle $L$ on $X$ with $\Phi$ induced from the multiplication of the tautological section of $\mathrm{Tot}(\mathcal{O}(\ell))$ on $L$. From Schwarzenberger \cite{MR0137712}, when $\ell=3$, $X$ is a non-singular K3 surface and when $\ell\geq4$, $X$ is a surface of general type.

Finally, consider a double cover $f:X\to\mathbb{P}^2$ branched along a singular curve $\rho\in H^0(\mathbb{P}^2,\mathcal{O}(2\ell))$ with $\ell\geq1$. In this case, $X$ is singular. Can one still construct a stable rank-two wild Vafa-Witten pair $(E,\Phi)$ with $\Phi\in H^0(\mathbb{P}^2,\mathrm{End}_0E(\ell))$ such that $\det\Phi=\rho$ using a spectral type correspondence? This can be done by adapting the techniques used to study Higgs bundles over Riemann surfaces with singular spectral curves (see \cite{MR4201792,MR3031827,MR3787793} and the references therein). This will be studied in future work.

\section{Moduli Space of Stable wild Vafa-Witten Bundles}\label{Sec4}

Let $c_1,c_2\in\mathbb{Z}$ and $d\in\mathbb{Z}^{\geq0}$. We denote $\mathcal{M}^{\mathrm{VW}}_{c_1,c_2,d}$ the moduli space of stable wild Vafa-Witten pairs $(E,\Phi)$ on $\mathbb{P}^2$ with $c_1(E)=c_1\cdot H,c_2(E)=c_2\cdot H^2$ and $\Phi\in H^0(\mathbb{P}^2,\mathrm{End}_0E(d))$. In this section, we consider moduli spaces with $c_1=r+s-1$ and $c_2=\frac{1}{2}(r(r-1)+s(s-1))$. Moreover, by the spectral correspondence, any stable pair $(E,\Phi)\in \mathcal{M}^{\mathrm{VW}}_{c_1,c_2,1}$ with smooth spectral cover and Chern classes  

$$c_1(E)=(r+s-1)\cdot H, \,\ c_2(E)=\left(\frac{1}{2}(r(r-1)+s(s-1))\right)\cdot H^2$$
for some integers $r,s\in\mathbb{Z}$ with $s\geq r$, is such that $E=E_{r,s}^\rho$ for some $\rho\in H^0(\mathbb{P}^2,\mathcal{O}(2))$. Given that the cohomology of $E_{r,s}^\rho$ is well understood, we get the following result by studying the deformation theory of such pairs: 

\begin{thm}\label{Thm16}
For $r,s\in\mathbb{Z}$ with $s\geq r$, let $c_1=r+s-1$ and $c_2=\frac{1}{2}(r(r-1)+s(s-1))$. The moduli space $\mathcal{M}_{c_1,c_2,1}^{\mathrm{VW}}$ is smooth with complex dimension 

$$\dim_\C \mathcal{M}_{c_1,c_2,1}^{\mathrm{VW}}=6.$$
\end{thm}

For $d>0$, not every stable pair $(E,\Phi)$ is such that $E=E_{r,s}^\rho$, and little is known about the cohomology of $E$ when $E\neq E_{r,s}^\rho$. We can nonetheless compute the dimension of the connected components of $\mathcal{M}_{c_1,c_2,d}^{\mathrm{VW}}$ containing pairs $(E_{r,s}^\rho,\Phi)$.

\begin{thm}\label{4.1}
    Let $d>0$ and $r,s\in\mathbb{Z}$ with $s\geq r$. Let $c_1=r+s-1$ and $c_2=\frac{1}{2}(r(r-1)+s(s-1))$. If $\mathcal{C}$ is a connected component of $\mathcal{M}_{c_1,c_2,d}^{\mathrm{VW}}$ containing pairs of the form $(E_{r,s}^\rho,\Phi)$, then

    $$\dim_\C \mathcal{C}=\frac{3}{2}d(d+3).$$
\end{thm}

\begin{rem}
    We are only considering the case when $d >0$ in Theorem \ref{4.1} as when $d=0$, the only traceless endomorphisms $\Phi$ for which the pair $(E,\Phi)$ is stable is $\Phi=0$ (see \cite[Lemma 3.2]{MR4204283}).
\end{rem}

\subsection{Hypercohomology Computations}

To determine the smooth loci of the moduli space of stable wild Vafa-Witten pairs whose underlying bundle is a Schwarzenberger bundle of type $\ell=1$, we consider the deformation theory of such pairs. By analyzing the associated spectral sequence arising from the deformation complex of a stable wild Vafa-Witten pair with underlying bundle a Schwarzenberger bundle $E_{r,s}^\rho$ of type $\ell=1$, we identify the first hypercohomology group of this complex with the Zariski tangent space to the moduli space $\mathcal{M}_{r,s,d}^{\mathrm{VW}}$ at the corresponding point. In particular, the dimension of the first hypercohomoloy coincides with the dimension of the Zariski tangent space at a stable pair $(E_{r,s}^\rho,\Phi)$. Let $d>0$, and fix a Higgs field $\Phi\in H^0(\mathrm{End}_0E_{r,s}^\rho(d))$. Consider the complex: 

$$0\to\mathrm{End}_0E_{r,s}^\rho\overset{-\land\Phi}{\to}\mathrm{End}_0E_{r,s}^\rho\wedge\mathcal{O}(d)\overset{-\land\Phi}{\to}\mathrm{End}_0E_{r,s}^\rho\wedge^2\mathcal{O}(d)\to\cdots.$$ Since $\mathcal{O}(d)$ is a line bundle, the above complex reduces to 

\begin{align}\label{sequence4.1}
    0\to \mathrm{End}_0E_{r,s}^\rho \overset{-\land\Phi}{\to}\mathrm{End}_0E_{r,s}^\rho(d)\to0.
\end{align}

In particular, the condition $\Phi\land\Phi=0$ is automatically satisfied, so the operation \\ $-\land\Phi$ defines a differential on the complex of \v{C}ech cochains associated to $\mathrm{End}_0E_{r,s}^\rho\otimes\wedge^\bullet\mathcal{O}(d)$. Denoting this complex of cochains by $C^\bullet(\mathrm{End}_0E_{r,s}^\rho\wedge^\bullet\mathcal{O}(d))$, we obtain the standard \v{C}ech resolution

$$0\to C^0(\mathrm{End}_0E_{r,s}^\rho\wedge^\bullet\mathcal{O}(d))\overset{\delta}{\to}C^1(\mathrm{End}_0E_{r,s}^\rho\wedge^\bullet\mathcal{O}(d))\overset{\delta}{\to} C^2(\mathrm{End}_0E_{r,s}^\rho\wedge^\bullet\mathcal{O}(d))\to\cdots ,$$

where $\delta$ denotes the ordinary \v{C}ech coboundry operator. We regard this as the zeroth sheet of the spectral sequence computing the hypercohomology of \ref{sequence4.1}, and denote by 

$$E^{p,q}_1=H^p_\delta(\mathrm{End}_0E_{r,s}^\rho\wedge^q\mathcal{O}(d))$$ 

the first-sheet terms, where $H_\delta$ indicates \v{C}ech cohomology with respect to $\delta$. Since $\wedge^q\mathcal{O}(d)=0$ for $q\geq2$, we immediately obtain $E^{p,q}_1=0$ for all such $q$. Moreover, as we are working on a complex surface, $H^p_\delta(\mathrm{End}_0E_{r,s}^\rho\wedge^q\mathcal{O}(d))=0$ for $p>2$; hence $E^{p,q}_1=0$ for $p>2$. Consequently, the only possible nontrivial terms on the $E_1$-page occur for $0\leq p\leq 2$ and $0\leq q \leq 1$. The second sheet of the associated spectral sequence is obtained by taking cohomology with respect to the differential $-\wedge\Phi$. Explicitly,

$$E^{p,q}_2=H^p_{-\land\Phi}(H^q_\delta(\wedge^\bullet\mathcal{O}(d)))=\frac{\ker(H^q_\delta(\mathrm{End}_0E_{r,s}^\rho\wedge^p\mathcal{O}(d))\overset{-\land\Phi}{\to}H^q_\delta(\mathrm{End}_0E_{r,s}^\rho\wedge^{p+1}\mathcal{O}(d)))}{\mathrm{im}(H^q_\delta(\mathrm{End}_0E_{r,s}^\rho\wedge^{p-1}\mathcal{O}(d))\overset{-\land\Phi}{\to}H^q_\delta(\mathrm{End}_0E_{r,s}^\rho\wedge^p\mathcal{O}(d)))}.$$

The terms $E_2^{p,q}$ thus encode the hypercohomology groups of the complex \ref{sequence4.1}, and fit into the exact sequence described in \cite[Proposition 3.2]{MR3285779}

\begin{align}\label{sequence4.2}
0\to E_2^{1,0}\to\mathbb{H}^1\to E_2^{0,1}\overset{d_2}{\to}E_2^{2,0}\to\mathbb{H}^2\to E_2^{1,1}\to 0.
\end{align}

The map $d_2:E_2^{p,q}\to E_2^{p+2,q-1}$ is a differential operator such that $d_2^2=0$, and understanding this map is needed in order to decipher the first hypercohomology group. For Schwarzenberger bundles, we have the following: 

$$E_2^{0,1}=\ker\left(H^1(\mathrm{End}_0E_{r,s}^\rho)\overset{-\land\Phi}{\to}H^1(\mathrm{End}_0E_{r,s}^\rho(d))\right)\subseteq H^1(\mathrm{End}_0E_{r,s}^\rho).$$

From Proposition \ref{3.4}, $H^1(\mathrm{End}_0E_{r,s}^\rho)=0$ for $s=r,r+1,r+2$, and thus, $\left(d_2:E^{0,1}_2\to E^{2,0}_2\right)\subseteq H^1(\mathrm{End}_0E_{r,s}^\rho)=0$ is the zero map. For $s=r+k$ with $k>2$, tt is not clear that $d_2=0$ as $H^1(\mathrm{End}_0E_{r,r+k}^\rho)$ is non-zero. One can show however that $d_2=0$ in this case as well (see Section 2.1 in \cite{SSR:11}). This then implies that there is a splitting 

$$0\to E^{1,0}_2\to\mathbb{H}^1\to E_2^{0,1}\to 0,$$ and so one can compute the first hypercohomology group of the complex as $\mathbb{H}^1=E^{1,0}_2\oplus E^{0,1}_2$. The two groups are 

$$E^{0,1}_2=\ker\left(H^1(\mathrm{End}_0E_{r,s}^\rho)\to H^1(\mathrm{End}_0E_{r,s}^\rho(d))\right)$$ and 

$$E^{1,0}_2=\frac{H^0(\mathrm{End}_0E_{r,s}^\rho(d))}{\mathrm{im}(H^0(\mathrm{End}_0E_{r,s}^\rho)\to H^0(\mathrm{End}_0E_{r,s}^\rho(d)))}.$$ For $s=r,r+1,r+2$, $H^1(\mathrm{End}_0E_{r,s}^\rho(d))=0$ for all $d\geq0$, and so $E^{0,1}_2$ vanishes. Thus,

$$\mathbb{H}^1\cong\frac{H^0(\mathrm{End}_0E_{r,s}^\rho(d))}{\mathrm{im}(H^0(\mathrm{End}_0E_{r,s}^\rho)\to H^0(\mathrm{End}_0E_{r,s}^\rho(d)))}$$ for these values of $s$ in which we can interpret as follows. The vector space $H^0(\mathrm{End}_0E_{r,s}^\rho(d))$ parametrizes trace-free global Higgs fields on the Schwarzenberger bundles $E_{r,s}^\rho$. The bottom image corresponds to the infinitesimal action of the automorphism group $\mathrm{Aut}(E_{r,s}^\rho)$ on these Higgs field via conjugation. Equivalently, the first hypercohomology group $\mathbb{H}^1$ of the deformation complex classifies infinitesimal deformations of stable pairs $(E_{r,s}^\rho,\Phi)$, or trace-free Higgs fields modulo conjugation by $\mathrm{Aut}(E_{r,s}^\rho)$. When the obstruction group $\mathbb{H}^2$ vanishes, the Zariski tangent space at a point $(E_{r,s}^\rho,\Phi)$ is identified with the quotient of $H^0(\mathrm{End}_0E_{r,s}^\rho(d))$ by this conjugation action.

We compute the dimension of this space in the cases $s=r$, $s=r+1$, or $s=r+2$ first then deal with the case $s=r+k$ with $k>2$ last. When $s=r+k$ for $k>2$, the group $E^{0,1}_2$ is non-zero and the first Hypercohomology group is $\mathbb{H}^1=E^{1,0}_2\oplus E^{0,1}_2$. The following known lemma is useful for these computations:

\begin{lem}\label{Lem19}
Let $V$ be a rank-two holomorphic vector bundle on $\mathbb{P}^2$. For any stable $(V,\Phi)$ with $\Phi\in H^0(\mathrm{End}V(d))$, the subspace of $H^0(\mathrm{End}V)$ consisting of endomorphisms that commute with $\Phi$ is generated by $\mathrm{Id}_V$.
\end{lem} 

See, for example, the proof of \cite[Proposition 7.1]{MR1085642} for a proof. We treat each case separately.

\subsection{Proofs of Theorems 4.1 and 4.2}
\subsection*{$\mathbf{s=r,r+1}$} Here, 

$$\dim\mathbb{H}^1=h^0(\mathrm{End}_0E_{r,r}^\rho(d))-\dim\left(\mathrm{im}(H^0(\mathrm{End}_0E_{r,r}^\rho)\to H^0(\mathrm{End}_0E_{r,r}^\rho(d)))\right)$$

and 

$$h^0(\mathrm{End}_0E_{r,r}^\rho(d))=\frac{1}{2}d^2+\frac{1}{2}d+4+4d+d^2=\frac{3}{2}d(d+3)+4.$$ For the image, we use Lemma \ref{Lem19}, which states that the dimension of the kernel of the map $H^0(\mathrm{End}_0E_{r,r}^\rho)\overset{\land\Phi}{\to}H^0(\mathrm{End}_0E_{r,r}^\rho(d))$ is one-dimensional, and so 

$$\dim\mathrm{im}((H^0(\mathrm{End}_0E_{r,r}^\rho)\to H^0(\mathrm{End}_0E_{r,r}^\rho(d)))=h^0(\mathrm{End}_0E_{r,r}^\rho)-1=5-1=4.$$ 

Thus, the dimension of the first hypercohomology group $\mathbb{H}^1$ of the spectral sequence when $s=r$ is 

$$\dim\mathbb{H}^1=\frac{3}{2}d(d+3).$$

Now, for $s=r+1$,

$$\dim\mathbb{H}^1=h^0(\mathrm{End}_0E_{r,r+1}^\rho(d))-\dim\left(\mathrm{im}(H^0(\mathrm{End}_0E_{r,r+1}^\rho)\to H^0(\mathrm{End}_0E_{r,r+1}^\rho(d)))\right)$$ and 

$$h^0(\mathrm{End}_0E_{r,r+1}^\rho(d))=\frac{3}{2}d(d+3)+3.$$ The kernel of $H^0(\mathrm{End}_0E_{r,r+1}^\rho)\to H^0(\mathrm{End}_0E_{r,r+1}^\rho(d))$ is one dimensional via Lemma \ref{Lem19}, and as a result

$$\dim\left(\mathrm{im}(H^0(\mathrm{End}_0E_{r,r+1}^\rho)\to H^0(\mathrm{End}_0E_{r,r+1}^\rho(d)))\right)=h^0(\mathrm{End}_0E_{r,r+1}^\rho)-1=4-1=3.$$ The dimension of the first hypercohomology group of the spectral sequence when $s=r+1$ is $\dim\mathbb{H}^1=\frac{3}{2}d(d+3)$.

Now these represent the dimension of the Zariski tangent space $T_{(E_{r,s}^\rho,\Phi)}\mathcal{M}_{c_1,c_2,d}^\mathrm{VW}$ at the point $(E_{r,s}^\rho,\Phi)$ for $s=r$ or $s=r+1$ given that the obstruction group $\mathbb{H}^2$ vanishes. From the exact sequence \ref{sequence4.2} and the fact that $d_2=0$, we have the short exact sequence 

$$0\to E_2^{2,0}\to\mathbb{H}^2\to E_2^{1,1}\to 0,$$ which states $\mathbb{H}^2\cong E_2^{2,0}\oplus E_2^{1,1}$. However, $E_2^{2,0}$ vanishes because $\wedge^2\mathcal{O}(d)=0$. Also,

$$E_2^{1,1}=\frac{H^1(\mathrm{End}_0E_{r,s}^\rho(d))}{\mathrm{im}(H^1(\mathrm{End}_0E_{r,s}^\rho)\overset{\land\Phi}{\to}H^1(\mathrm{End}_0E_{r,s}^\rho(d)))},$$ but $H^1(\mathrm{End}_0E_{r,s}^\rho(d))=0$ for $s=r$ or $s=r+1$. Thus $E_2^{1,1}=0$ and so the obstruction vanishes. This shows that the points $(E_{r,s}^\rho,\Phi)$ in $\mathcal{M}_{c_1,c_2,d}^{\mathrm{VW}}$ for $s=r$ or $s=r+1$ are smooth points with expected dimension $\frac{3}{2}d(d+3)$.

\subsection*{$\mathbf{s=r+2}$} In this case, the computations are simplified since for $s\geq r+2$ the bundles $E_{r,s}^\rho$ are simple and so automatically stable. Thus, 

$$\mathbb{H}^1=E_2^{1,0}=\frac{H^0(\mathrm{End}_0E_{r,s}^\rho(d))}{\mathrm{im}(H^0(\mathrm{End}_0E_{r,s}^\rho)\to H^0(\mathrm{End}_0E_{r,s}^\rho(d)))}=H^0(\mathrm{End}_0E_{r,s}^\rho(d)).$$  Now from Proposition \ref{Prop11}

\begin{align*}
    h^0(\mathrm{End}_0E_{r,r+2}^\rho(d)) = \frac{3}{2}d(d+3),
\end{align*}

and so $\dim \mathbb{H}^1=\frac{3}{2}d(d+3)$. In this case, $\mathbb{H}^2$ also vanishes since $H^1(\mathrm{End}_0E_{r,r+2}^\rho(d))=0$. Thus in $\mathcal{M}^{\mathrm{VW}}_{c_1,c_2,d}$ the point $(E^\rho_{r,r+2},\Phi)$ is smooth with the expected dimension of the Zariski tangent space being
$\frac{3}{2}d(d+3)$. To obtain the result in Theorem \ref{Thm16} for the case $s=r,r+1,r+2$, set $d=1$ in $\frac{3}{2}d(d+3)$.

\subsection*{$\mathbf{s=r+k} \,\ \textbf{for} \,\ \mathbf{k>2}$}

For $s>r+2$, there is still a splitting of \ref{sequence4.2} since $d_2=0$ and as such, the first and second hypercohomology groups are 

$$\mathbb{H}^1=H^0(\mathrm{End}_0E_{r,r+k}^\rho(d))\oplus\ker\left(H^1(\mathrm{End}_0E_{r,r+k}^\rho)\overset{-\land\Phi}{\to} H^1(\mathrm{End}_0E_{r,r+k}^\rho(d))\right)$$ and 

$$\mathbb{H}^2=\frac{H^1(\mathrm{End}_0E_{r,r+k}^\rho(d))}{\mathrm{im}(H^1(\mathrm{End}_0E_{r,r+k}^\rho)\overset{-\land\Phi}{\to}H^1(\mathrm{End}_0E_{r,r+k}^\rho(d)))},$$

Respectively. From Proposition \ref{3.5} we can see that for $d>1$, the obstruction $\mathbb{H}^2$ vanishes and $(E_{r,r+k}^\rho,\Phi)$ is a smooth point with expected dimension 

\begin{align*} 
\dim\mathbb{H}^1=h^0(\mathrm{End}_0E_{r,r+k}^\rho(d))+h^1(\mathrm{End}_0E_{r,r+k}^\rho) &= \frac{3}{2}d(d+3)+4-k^2+k^2-4 \\
&= \frac{3}{2}d(d+3).
\end{align*}

If $d=1$, Proposition \ref{3.5} states that $h^1(\mathrm{End}_0E_{r,r+k}^\rho(1))=k^2-9$, so it is not trivially true that $\mathbb{H}^2=0$ as the other cases. To show that the obstruction $\mathbb{H}^2$ vanishes, we use the following lemma from \cite[Lemma 2.2]{MR3285779} to compute the image of $H^1(\mathrm{End}_0E_{r,s}^\rho)\overset{-\land\Phi}\to H^1(\mathrm{End}_0E_{r,s}^\rho(1))$; 

\begin{lem}\label{lem20}
    Let $X$ be a smooth complex manifold. Let $(E,\Phi)$ be a regular Higgs bundles and $L$ a line bundle both on $X$ with $\Phi\in H^0(X,\mathrm{End}_0E\otimes L)$. Then, we have a short exact sequence
    $$0\to L^*\overset{\Phi}{\to}\mathrm{End}_0E\overset{-\land\Phi}{\to}Q\to0,$$ where $Q$ is the sheaf theoretic image of $-\land\Phi$ in $\mathrm{End}_0E\otimes L$. If one views $\Phi$ as an element of $\Gamma(\mathrm{Hom}(\mathrm{End}_0E\otimes L,L\otimes L))$, then if for some $c\in \mathbb{C}$ we have $\Phi^*=c\Phi$, then $\ker\Phi\cong Q$
\end{lem}

\begin{proof}
This is Lemma 2.2 in \cite{MR3285779}
\end{proof}

In the lemma, a regular Higgs bundles means that the spectral cover of $\Phi$ is smooth, and $\Phi^*$ is defined as the dual element of $\Phi$ in the vector space $\Gamma(\mathrm{Hom}(E^*,E^*\otimes L))$. We use the lemma for $\Phi\in H^0(\mathrm{End}_0E_{r,s}^\rho(1))$ with $s=r+k$ for $k>2$ to show that $\mathbb{H}^2$ is still zero in this case. From the lemma, we have the short exact sequence;

\begin{align}\label{5.2}
0\to\mathcal{O}(-1)\overset{\Phi}{\to}\mathrm{End}_0E_{r,r+k}^\rho\overset{-\land\Phi}{\to}Q\to 0.
\end{align}

Since $\Phi$ and $\Phi^*$ can be both viewed as living inside $H^0(\mathrm{End}_0E_{r,r+k}^\rho(1))$, combined with the fact that $h^0(\mathrm{End}_0E_{r,r+k}^\rho(1))=1$, we have that $\Phi^*=c\Phi$. From the lemma, this implies that we can view $Q\cong\ker\Phi$ which gives us the following short exact sequence;

\begin{align}\label{5.3}
    0\to Q\to \mathrm{End}_0E_{r,r+k}^\rho(1)\overset{\Phi}{\to}\mathcal{O}(2)\to0.
\end{align}

The long exact sequence in cohomology of \ref{5.2} tells us that $H^0(Q)=0$ and $H^1(\mathrm{End}_0E_{r,r+k}^\rho)=H^1(Q)$. Combining this with the long exact sequence in cohomology from \ref{5.3}, we arrive at the following short exact sequence;

$$0\to H^0(\mathrm{End}_0E_{r,r+k}^\rho(1))\to H^0(\mathcal{O}(2))\to H^1(\mathrm{End}_0E_{r,r+k}^\rho)\overset{-\land\Phi}{\to}H^1(\mathrm{End}_0E_{r,r+k}^\rho(1))\to 0.$$ This shows that the mapping $-\land\Phi: H^1(\mathrm{End}_0E_{r,r+k}^\rho)\to H^1(\mathrm{End}_0E_{r,r+k}^\rho(1))$ is surjective. With this, going back and recalling that

$$\mathbb{H}^2=\frac{H^1(\mathrm{End}_0E_{r,r+k}^\rho(1))}{\mathrm{im}(H^1(\mathrm{End}_0E_{r,r+k}^\rho)\overset{-\land\Phi}{\to}H^1(\mathrm{End}_0E_{r,r+k}^\rho(1)))}$$ which vanishes from the above, showing that for $d=1$, the point $(E_{r,r+k}^\rho,\Phi)$ with $k>2$ is a smooth point in the moduli space with expected dimension $\mathrm{dim}\mathbb{H}^1$. Recall that in this case

$$\mathbb{H}^1=H^0(\mathrm{End}_0E_{r,r+k}^\rho(1))\oplus\ker\left(H^1(\mathrm{End}_0E_{r,r+k}^\rho)\overset{-\land\Phi}{\to} H^1(\mathrm{End}_0E_{r,r+k}^\rho(1))\right),$$ which for $k=3$, $h^1(\mathrm{End}_0E_{r,r+3}^\rho(1))=0$ and so $\dim\mathbb{H}^1=h^0(\mathrm{End}_0E_{r,r+k}^\rho(1))+5=6$.

When $k>3$, using the fact that the mapping $-\land\Phi: H^1(\mathrm{End}_0E_{r,r+k}^\rho)\to H^1(\mathrm{End}_0E_{r,r+k}^\rho(1))$ is surjective

$$\dim\ker\left(H^1(\mathrm{End}_0E_{r,r+k}^\rho)\overset{-\wedge\Phi}{\to}H^1(\mathrm{End}_0E_{r,r+k}^\rho(d))\right)=h^1(\mathrm{End}_0E_{r,sr+k}^\rho)-h^1(\mathrm{End}_0E_{r,s}^\rho(1))=5.$$

Thus, for $d=1$, with $s=r+k$ with $k>3$, the point $(E_{r,r+k}^\rho,\Phi)$ is smooth with expected dimension $\dim\mathbb{H}^1=h^0(\mathrm{End}_0E_{r,r+k}^\rho(1))+5=1+5=6$.
\qed

\section{Fixed Points}\label{Sec5}

Let $c_1,c_2\in\mathbb{Z}$ and $d\in\mathbb{Z}^{\geq0}$. We denote $\mathcal{M}^{\mathrm{VW}}_{c_1,c_2,d}$ the moduli space of stable wild Vafa-Witten pairs $(E,\Phi)$ on $\mathbb{P}^2$ with $c_1(E)=c_1\cdot H,c_2(E)=c_2\cdot H^2$ and $\Phi\in H^0(\mathbb{P}^2,\mathrm{End}_0E(d))$. The multiplicative group $\mathbb{C}^*$ acts naturally on $\mathcal{M}_{c_1,c_2,d}^{\mathrm{VW}}$ by scaling the Higgs field, that is, 

$$\Phi\mapsto c\Phi, \,\ c\in\mathbb{C}^*.$$

The fixed point locus of this action is compact, and, following Tanaka and Thomas \cite{MR4158461}, one can define a numerical invariant on $\mathcal{M}_{c_1,c_2,d}^\mathrm{VW}$ by applying virtual localization in the sense of \cite{MR1666787}. This invariant is realized as a virtual Euler class integrated over the $\mathbb{C}^*$-fixed locus, and may be interpreted heuristically as a virtual count of stable pairs in the moduli space. In particular, the global geometry of $\mathcal{M}_{c_1,c_2,d}^\mathrm{VW}$ can be studied through the structure of these fixed loci. The resulting theory extends the Donaldson-Thomas invariants of \cite{MR1634503} to the setting of wild Vafa-Witten pairs, in which the Higgs field takes values in $\mathcal{O}(d)$ for $d\geq0$. The computation of the Tanaka-Thomas invariants of $\mathcal{M}_{c_1,c_2,d}^\mathrm{VW}$ will be done in a future work; for now, we describe its $\mathbb{C}^*$-fixed locus. Concretely, let $E$ be a rank-two torsion-free sheaf on $\mathbb{P}^2$, and let $\Phi\in H^0(\mathrm{End}_0E(d))$ for some $d\geq0$. Suppose $(E,\Phi)$ is a $\mathbb{C}^*$-fixed point. Then, by Theorem \ref{Thm8}, we may decompose 

$$E=E_1\oplus E_2, \,\ \Phi=\begin{pmatrix} 0 & 0 \\ \phi & 0 \end{pmatrix},$$ where $\phi: E_1\to E_2(d)$ is a non-zero morphism. Since each $E_i$ is a rank-one torsion-free sheaf, we may write $E_i=I_{Z_i}\otimes\mathcal{O}(m_i)$, for some $m_i\in\mathbb{Z}$, where $I_{Z_i}$ denotes the ideal sheaf of a zero-dimensional subscheme $Z_i\subseteq\mathbb{P}^2$ for $i=1,2$. The nontriviality of $\phi$ requires that $Z_2\subseteq Z_1$ as subschemes of $\mathbb{P}^2$. Without loss of generality, we may assume $m_1\geq m_2$. Stability then imposes the relation $m_2=m_1-j$ for $0\leq j\leq d$. Accordingly, we can rewrite the bundle as

$$E=\mathcal{O}(m_1)\otimes(I_{Z_1}\oplus I_{Z_2}(-j))$$ and the Higgs field $\phi\in H^0(\mathrm{Hom}(I_{Z_1},I_{Z_2})\otimes\mathcal{O}(d-j))$. Let 

$$n:=\#|Z_1|+\#|Z_2|$$

denote the total length of the subschemes. A direct computation yields the Chern classes 

$$c_1(E)=2m_1-j, \,\ c_2(E)=m_1(m_1-j)+n.$$ The following proposition, which follows from \cite[Lemma 8.3]{MR4158461}, provides a classification of all possible fixed points of the $\mathbb{C}^*$-action. In what follows, we set $m_1=m$ for notational simplicity.

\begin{prop}\label{6.1}
Let $d\geq 0$, $0\leq j\leq d$, and $m\geq0$. The fixed points of the $\C^*$-action in the moduli space of stable wild Vafa-Witten pairs on $\mathbb{P}^2$ are of the form 

$$E=\mathcal{O}(m)\otimes (I_{Z_1}\oplus I_{Z_2}(-j)), \,\ \Phi=\begin{pmatrix} 0 & 0 \\ \iota\otimes s & 0 \end{pmatrix}$$ where $\iota: I_{Z_1} \hookrightarrow{} I_{Z_2} $ is the inclusion and $s\in H^0(\mathcal{O}(d-j))$.
   
\end{prop}

\begin{proof}
We already demonstrated how we can write our bundle as $$E=\mathcal{O}(m)\otimes (I_{Z_1}\oplus I_{Z_2}(-j)).$$ We show that we can think of $\phi\in H^0(\mathrm{Hom}(I_{Z_1},I_{Z_2})\otimes\mathcal{O}(d-j))$ globally as $\iota\otimes s$ for $\iota: I_{Z_1}\hookrightarrow\ I_{Z_2}$ the inclusion and $s\in H^0(\mathcal{O}(d-j))$. Consider the short exact sequence for the ideal $I_{Z_2}$, 

$$0\to I_{Z_2}\to\mathcal{O}\to\mathcal{O}/I_{Z_2}\to 0.$$ Applying the functor $\mathrm{Hom}(I_{Z_1},\cdot)$ gives 

$$0\to \mathrm{Hom}(I_{Z_1},I_{Z_2})\to\mathrm{Hom}(I_{Z_1},\mathcal{O})\to Q\to 0$$ where $Q$ is the quotient sheaf. Tensoring by $\mathcal{O}(d-j)$, which preserves exactness, we arrive at

$$0\to \mathrm{Hom}(I_{Z_1},I_{Z_2})\otimes\mathcal{O}(d-j)\to\mathrm{Hom}(I_{Z_1},\mathcal{O})\otimes\mathcal{O}(d-j)\to Q\otimes\mathcal{O}(d-j)\to0.$$ From the long exact sequence in cohomology we arrive at the fact that $$h^0(\mathrm{Hom}(I_{Z_1},I_{Z_2})\otimes\mathcal{O}(d-j))\leq h^0(\mathrm{Hom}(I_{Z_1},\mathcal{O})\otimes\mathcal{O}(d-j)),$$ in which we know that $h^0(\mathrm{Hom}(I_{Z_1},\mathcal{O})\otimes\mathcal{O}(d-j))=h^0(\mathcal{O}(d-j))$. Since $Z_2\subset Z_1$ we know that $\mathrm{Hom}(I_{Z_1},I_{Z_2})$ contains at least one global element, namely the inclusion $\iota: I_{Z_1} \hookrightarrow I_{Z_2}$. In particular, given $s\in H^0(\mathcal{O}(d-j))$, the product $\iota\otimes s\in H^0(\mathrm{Hom}(I_{Z_1},I_{Z_2})\otimes\mathcal{O}(d-j))$. Since we have an $h^0(\mathcal{O}(d-j))$ worth of $s$ to pick from and that $$h^0(\mathrm{Hom}(I_{Z_1},I_{Z_2})\otimes\mathcal{O}(d-j))\leq h^0(\mathcal{O}(d-j)),$$ we must have that all elements of $H^0(\mathrm{Hom}(I_{Z_1},I_{Z_2})\otimes\mathcal{O}(d-j))$ are of the form $\iota\otimes s$ for some $s\in H^0(\mathcal{O}(d-j))$. 

\end{proof}

We illustrate this with a few examples based on the total number of points in the subschemes.

\begin{exmp} 
$n=0$: When $\#|Z_1|+\#|Z_2|=0$, we are in the situation where $Z_1=Z_2=\varnothing$ are both empty as sets. This forces $E=\mathcal{O}(m)\oplus \mathcal{O}(m-j)$ and $\phi:\mathcal{O}(m)\to\mathcal{O}(m-j+d)$ or $\phi\in H^0(\mathcal{O}(d-j))$. 
\end{exmp}

\begin{exmp}
$n=1$: When $\#|Z_1|+\#|Z_2|=1$, since we have $Z_2\subseteq Z_1$, we are in the situation where $Z_1=\{p\}$ and $Z_2=\varnothing$. In this case, the bundle is $$E=\mathcal{O}(m)\otimes (I_p\oplus\mathcal{O}(-j))$$ where $I_p$ is the ideal sheaf supported at a point $p\in\mathbb{P}^2$. The Higgs field is $\phi\in\mathrm{Hom}(I_p(m),\mathcal{O}(m-j+d))$ which we can interpreted as $\phi=\iota_{I_p}\otimes s$ where $s\in H^0(\mathcal{O}(d-j))$ and $\iota:I_p\hookrightarrow \mathcal{O}$ is the inclusion.
\end{exmp}

\begin{exmp}
$n=2$: When $\#|Z_1|+\#|Z_2|=2$ we have two cases. Either $Z_1=\{p_1,p_2\}$ and $Z_2=\varnothing$ or $Z_1=Z_2=\{p\}$. In the former, the bundle is $$E=\mathcal{O}(m)\otimes(I_{Z_1}\oplus\mathcal{O}(-j))$$ and the Higgs field is $\phi\in\mathrm{Hom}(I_{Z_1}(m),\mathcal{O}(m-j+d))$ which can be interpreted as $\phi=\iota_{I_{Z_1}}\otimes s$ with $s\in H^0(\mathcal{O}(d-j))$ where $\iota_{Z_1}:Z_1\hookrightarrow\mathcal{O}$ is the inclusion. In the latter case when $Z_1=Z_2=\{p\}$, the bundle is $$E=\mathcal{O}(m)\otimes(I_p\oplus I_p(-j)).$$ Again, $I_p$ is the ideal sheaf supported on a point $p\in\mathbb{P}^2$. The Higgs field is a section  

$$\phi\in\mathrm{Hom}(I_p(m),I_p(m-j+d))$$ 

which can be interpreted as $\phi=\mathrm{Id}_{I_p}\otimes s$ with $s\in H^0(\mathcal{O}(d-j))$ (here $\iota:I_p \to I_p$ is the identity). 
\end{exmp}

When trying to compute the Tanaka-Thomas invariants for the moduli space of stable wild pairs on $\mathbb{P}^2$, it's better to compute the fixed points by first fixing a Chern class, usually the second Chern class. We demonstrate by showing two examples by computing fixed points for low Chern classes.

\begin{exmp}
When $c_2(E)\leq0$, we have that $m(m-j)+\ell\leq0$ or $\ell\leq-m(m-j)$. Since $\ell\geq0$, we must have $m(m-j)\leq0$. So either $m=0$ or $m=j$ with $\ell=0$ or $\ell=-m(m-j)$ for $0<m< j$. In the former, $E=\mathcal{O}\oplus \mathcal{O}(-j)$ for $m=0$ or $E=\mathcal{O}(j)\oplus\mathcal{O}$ for $m=j$ (the two are equivalent up to tensoring by the line bundle $\mathcal{O}(\pm j)$).
\end{exmp}

\begin{rem} 
When looking to compute the Tanaka-Thomas invariants for $c_2(E)<0$, when $d=1$ the component of the $\mathbb{C}^*$-fixed locus is empty. Since there are no stable pairs with $\phi\neq0$ and $c_2(E)<0$. Also, for $\phi=0$, there are no stable bundles with $c_2(E)<0$. However, when we increase $d$, say $d=2$, there will be contributions when $\phi\neq 0$. In general, for $c_2(E)<0$ and $d\geq2$, the fixed locus will be non-empty.
\end{rem}

\begin{exmp}
When $c_2(E)=1$, we have that $m(m-j)+\ell =1$ or $1-m(m-j)=\ell$. Since $\ell\geq 0$, this implies $m(m-j) \leq 1$. Again, we could have $m=0$ or $m=j$ in which $\ell=1$ in both cases. We already examined the case when $\ell=1$, this corresponds to when $E=\mathcal{O}(m)\otimes (I_p\oplus\mathcal{O}(-j))$ with $p\in\mathbb{P}^2$ a point. For $m=0$ we have $E=I_p\oplus \mathcal{O}(-j)$ and for $m=j$ we have $E=I_p(j)\oplus\mathcal{O}$ (again these are equivalent up to tensoring by a line bundle). For $m\neq 0,j$ we would require $0<m<j$ in which $\ell=1-m(m-j)$ and one can compute what form the bundle $E$ would take for specific values of $d,j$ and $m$.
\end{exmp}

\bibliography{biblio}
\bibliographystyle{acm}

\end{document}